\documentclass{amsart}

\usepackage{url}
\usepackage{amsfonts}
\usepackage{amssymb}
\usepackage{amsmath}
\usepackage{amsthm}
\usepackage{mathrsfs}
\usepackage{enumitem}

\newcommand{\kf}{\mathbf{k}}
\newcommand{\QSh}{\mathcal{H}_{\mathrm{qsh}}}
\newcommand{\Sh}{\mathcal{H}_{\mathrm{sh}}}
\newcommand{\Lie}{\mathfrak{g}}
\newcommand{\Free}{\mathfrak{f}}
\newcommand{\GG}{\mathcal{G}}
\newcommand{\Seq}{\operatorname{Seq}}
\newcommand{\wt}{\operatorname{wt}}
\newcommand{\Gal}{\operatorname{Gal}_{\sigma}}
\newcommand{\Spec}{\operatorname{Spec}}
\newcommand{\Gm}{\mathbb{G}_m}

\theoremstyle{plain}
\newtheorem{prop}{Proposition}
\newtheorem{theorem}{Theorem}
\newtheorem{lemma}{Lemma}
\newtheorem{cor}{Corollary}

\theoremstyle{definition}

\newtheorem{example}{Example}

\theoremstyle{remark}
\newtheorem{rem}{Remark}

\begin{document}

\title[Harmonic sums and difference Galois groups]
{Harmonic sums and the Galois group of the Mellin--KZ difference equation}
\author{Nikita Markarian}

\begin{abstract}
We identify the universal Galois group of the difference equation obtained
by the Mellin transform of the Knizhnik--Zamolodchikov equation with the
prounipotent group associated with the transport Hopf algebra of Deligne
and Terasoma studied in \cite{Markarian}.
At each finite weight, we realize the Picard--Vessiot ring by finite
multiple harmonic sums.
We also interpret the classical Gamma-corrected projected associator
as the comparison between finite and regularized asymptotic fibers
in this realization.
\end{abstract}

\email{nikita.markarian@gmail.com}

\date{}

\address{UMR 7501, Universit\'e de Strasbourg,
7 rue Ren\'e Descartes,
67084 Strasbourg Cedex, France}

\maketitle

\section*{Introduction}

In 
\cite{Markarian}, we developed the approach of Deligne and Terasoma
\cite{DeligneTerasoma,Terasoma} to double shuffle relations.  Its
central object is the transport algebra $W$, which describes the
category of unipotent perverse sheaves on $\Gm$ smooth outside $1$,
modulo sheaves smooth on all of $\Gm$.  Multiplicative convolution
gives the tensor product, and vanishing cycles at $1$ give a fiber
functor.  Since the Mellin transform takes multiplicative convolution
to multiplication, it is natural to ask how this algebra is related
to the difference equation obtained from the KZ equation.

Consider the KZ equation
\begin{equation}
 \left(\partial_z-\frac{A}{z}-\frac{B}{1-z}\right)G(z)=0.
 \label{KZ-intro}
\end{equation}
The Mellin transform on $(0,1)$ gives
\begin{equation}
 F(s+1)=(s+1+A-B)^{-1}(s+A)F(s).
 \label{Mellin-KZ-intro}
\end{equation}
After a rational gauge transformation, the universal formal version
of this equation has an ordered-product solution whose coefficients
are finite multiple harmonic sums.  The completed graded dual of
their quasi-shuffle Hopf algebra \cite{Hoffman} is the harmonic
presentation of $W$.  This gives an explicit tensor realization of
the transport algebra by rational difference systems.

The Galois calculation is a short consequence of standard facts about
unipotent groups and rational difference equations.  It shows that
each finite-weight universal system has the full unipotent group
prescribed by the harmonic Hopf algebra (Theorem~\ref{main-theorem}).
Passing to the inverse limit gives its known free prounipotent group.
The same argument computes the groups of the Mellin--KZ systems with
simultaneously strictly upper-triangular residues
(Theorem~\ref{fixed-matrix-theorem}).

Universal difference systems with this harmonic Hopf algebra occur
in Joyner's work \cite{Joyner}; our conventions are compared in
Remark~\ref{rem:Joyner}.  The Galois argument also recovers the
positive-index case of Ablinger and Schneider's independence theorem
for harmonic sums \cite{AblingerSchneider}.  This concerns finite sums
as sequences and makes no assertion about their limiting multiple
zeta values.

Over $\mathbb C$, the regularized ordered product compares a finite
fiber at $1$ with a regularized asymptotic fiber at infinity, after a
shift of the independent variable.  Its identification with the
Gamma-corrected projected KZ associator is the classical
regularization formula of Costermans and Hoang Ngoc Minh
\cite[Theorem~6]{CostermansMinh}, related to the comparison of Ihara,
Kaneko and Zagier \cite[Section~2]{IKZ}.  Together with the reference
comparison supplied by the constant trivialization, this gives two
tensor isomorphisms between the fibers.  These fiber functors are
copies of the forgetful functor in the difference-equation model;
their identification with the geometric fiber functors on the
perverse-sheaf quotient remains to be established.

\medskip
\noindent\textit{Acknowledgements.}
We thank P.~Etingoff for fruitful discussions.

\subsection*{Notations}

Throughout the paper, $\kf$ is an algebraically closed field of
characteristic zero.  The base difference field is $\kf(s)$ with
$\sigma(s)=s+1$.

The ring $\Seq(\kf)$ is the ring of sequences with values in $\kf$, where two
sequences are identified if they differ at only finitely many places.  The
shift is
\[
 \sigma\langle a(0),a(1),a(2),\ldots\rangle
 =\langle a(1),a(2),a(3),\ldots\rangle.
\]
The field $\kf(s)$ embeds in $\Seq(\kf)$ by evaluation at sufficiently large
integers.

The weight of the letter $y_m$ and of the generator $e_m$ is $m$.  The
weight of a word is the sum of the weights of its letters.  Completions are
always taken with respect to the weight filtration.

\section{The quasi-shuffle group}
\label{sec:quasi-shuffle}

Let $Y=\{y_1,y_2,\ldots\}$ and let $\kf\langle Y\rangle$ be the vector space
spanned by words in the alphabet $Y$.  The quasi-shuffle product is defined
recursively by
\begin{equation}
 (y_au)*(y_bv)
 =y_a(u*(y_bv))+y_b((y_au)*v)+y_{a+b}(u*v),
 \label{quasi-shuffle}
\end{equation}
with the empty word as unit.  Together with the deconcatenation coproduct
\begin{equation}
 \Delta(w)=\sum_{uv=w}u\otimes v,
 \label{deconcatenation}
\end{equation}
this defines a connected graded commutative Hopf algebra, denoted by
$\QSh$.

Fix a total order on $Y$ and use the induced lexicographic order on words.
A nonempty word $w$ is a Lyndon word if it is strictly smaller than each of
its proper nontrivial suffixes.

\begin{theorem}[Hoffman \cite{Hoffman}]
The quasi-shuffle algebra $\QSh$ is a polynomial algebra on the Lyndon
words.  Moreover, there is a weight-preserving Hopf algebra isomorphism
\begin{equation}
 \exp_H\colon \Sh\longrightarrow\QSh
 \label{Hoffman-isomorphism}
\end{equation}
from the shuffle Hopf algebra on the same alphabet.
\end{theorem}

For a word $w=a_1\cdots a_r$ and a composition
$I=(i_1,\ldots,i_k)$ of $r$, let $I[w]$ be the word obtained by replacing
each consecutive block of $i_j$ letters by their bracket, where
$[y_a,y_b]=y_{a+b}$.  Hoffman's isomorphism and its inverse are given by
\begin{align}
 \exp_H(w)
 &=\sum_{I\vDash r}\frac{1}{i_1!\cdots i_k!}I[w],
 \label{Hoffman-exp}\\
 \log_H(w)
 &=\sum_{I\vDash r}
 \frac{(-1)^{r-k}}{i_1\cdots i_k}I[w].
 \label{Hoffman-log}
\end{align}
For example,
\[
 \exp_H(y_ay_b)=y_ay_b+\frac12y_{a+b}.
\]

Let $X_m$ be dual to $y_m$.  The completed graded dual of $\QSh$ has
underlying algebra
\begin{equation}
 \QSh^\vee=\kf\langle\!\langle X_1,X_2,\ldots\rangle\!\rangle.
 \label{dual-algebra}
\end{equation}
Its coproduct is determined by
\begin{equation}
 \delta(X_m)=X_m\otimes1+1\otimes X_m
 +\sum_{a+b=m}X_a\otimes X_b.
 \label{dual-coproduct}
\end{equation}
Consequently, the series
\begin{equation}
 X(t)=1+\sum_{m\geq1}X_mt^m
 \label{group-like-X}
\end{equation}
is group-like:
\begin{equation}
 \delta(X(t))=X(t)\otimes X(t).
 \label{X-group-like}
\end{equation}

Define $L_m$ by
\begin{equation}
 L(t)=\log X(t)=\sum_{m\geq1}L_mt^m.
 \label{primitive-L}
\end{equation}
Since the logarithm of a group-like element is primitive, every $L_m$ is
primitive.  Explicitly,
\begin{equation}
 L_m=\sum_{r=1}^m\frac{(-1)^{r-1}}r
 \sum_{i_1+\cdots+i_r=m}X_{i_1}\cdots X_{i_r}.
 \label{Lm-formula}
\end{equation}
In small weights,
\begin{align}
 L_1&=X_1,\notag\\
 L_2&=X_2-\frac12X_1^2,\notag\\
 L_3&=X_3-\frac12(X_1X_2+X_2X_1)+\frac13X_1^3.
 \label{small-L}
\end{align}

Let $\Free=\Free(e_1,e_2,\ldots)$ be the free graded Lie algebra with
$\wt(e_m)=m$.  Formulae \eqref{Hoffman-isomorphism}--\eqref{primitive-L}
give an isomorphism of complete Hopf algebras
\begin{equation}
 \widehat{U(\Free)}\longrightarrow\QSh^\vee,
 \qquad e_m\longmapsto L_m.
 \label{free-dual-isomorphism}
\end{equation}
Thus, the affine group scheme represented by $\QSh$ is the free
prounipotent group
\begin{equation}
 \GG^{\mathrm{un}}=\exp\widehat{\Free}.
 \label{universal-group}
\end{equation}

For $N\geq1$, set
\begin{equation}
 \Free_{\leq N}=\Free\Big/\bigoplus_{q>N}\Free_q,
 \qquad
 \GG_N=\exp(\Free_{\leq N}).
 \label{finite-groups}
\end{equation}
The Lie algebra $\Free_{\leq N}$ is finite-dimensional and nilpotent.

\begin{prop}
Let $\QSh^{\leq N}$ be the subalgebra of $\QSh$ generated by the Lyndon words of
weight at most $N$.  It is a Hopf subalgebra and
\begin{equation}
 \QSh^{\leq N}=\mathcal{O}(\GG_N).
 \label{coordinate-GN}
\end{equation}
In particular, it is a polynomial algebra in finitely many variables.
\end{prop}

\begin{proof}
The coproduct of a word is a sum of its prefixes tensored with its suffixes.
Every proper prefix and suffix has smaller weight.  By Hoffman's theorem,
each of them is a polynomial in Lyndon words of no greater weight.  Thus,
$\QSh^{\leq N}$ is a Hopf subalgebra.

Let $H=\Spec(\QSh^{\leq N})$.  The subalgebra $\QSh^{\leq N}$ contains
every homogeneous component of $\QSh$ of weight at most $N$, and its
augmentation ideal modulo its square has a basis given by the Lyndon
generators of those weights.  Dually, the induced map
$\widehat{\Free}\to\operatorname{Lie}(H)$ is an isomorphism in weights
at most $N$ and has zero target in higher weights.  Its kernel is
therefore precisely the terms of weight greater than $N$.  The
correspondence between unipotent groups and nilpotent Lie algebras in
characteristic zero identifies $H$ with $\GG_N$.
\end{proof}

\section{The universal difference system}
\label{sec:universal-system}

We recall the finite harmonic sums and their generating series
\cite{CostermansMinh}, in the conventions needed for the
Picard--Vessiot construction.

For a nonempty word $w=y_{m_1}\cdots y_{m_r}$, define the finite multiple harmonic
sum
\begin{equation}
 H_w(M)=H_{m_1,\ldots,m_r}(M)
 =\sum_{M\geq n_1>\cdots>n_r\geq1}
 \frac{1}{n_1^{m_1}\cdots n_r^{m_r}}.
 \label{harmonic-sum}
\end{equation}
For the empty word, or empty index tuple, we write
$H_{\varnothing}(M)=1$.  In particular, the depth-one sum
$H_1(M)=\sum_{n=1}^M1/n$ is distinct from this constant.

\begin{prop}
The map
\begin{equation}
 \QSh\longrightarrow\Seq(\kf),
 \qquad w\longmapsto\langle H_w(M)\rangle_{M\geq0},
 \label{harmonic-character}
\end{equation}
is an algebra homomorphism.
\end{prop}

\begin{proof}
The product of two sums is divided into the regions in which the two outer
indices are unequal and the diagonal on which they are equal.  These three
terms give the three summands in \eqref{quasi-shuffle}.  Induction on the
sum of the depths proves the statement.
\end{proof}

For example,
\begin{equation}
 H_aH_b=H_{a,b}+H_{b,a}+H_{a+b}.
 \label{first-quasi-shuffle}
\end{equation}
The shift satisfies
\begin{equation}
 \sigma(H_{m_1,\ldots,m_r})
 =H_{m_1,\ldots,m_r}
 +\frac{1}{(s+1)^{m_1}}H_{m_2,\ldots,m_r}.
 \label{harmonic-shift}
\end{equation}

Let $x_N(t)$ be the image in $\GG_N$ of the group-like series $X(t)$ from
\eqref{group-like-X}.  In the primitive coordinates
\eqref{free-dual-isomorphism}, it is
\begin{equation}
 x_N(t)=\exp\left(\sum_{m=1}^Ne_mt^m\right).
 \label{xN}
\end{equation}
Choose a faithful finite-dimensional representation
\begin{equation}
 \rho_N\colon\GG_N\longrightarrow\operatorname{GL}(V_N).
 \label{faithful-representation}
\end{equation}
The universal difference system of weight $N$ is
\begin{equation}
 \sigma(Y)=\rho_N(a_N(s))Y,
 \qquad
 a_N(s)=x_N\left(\frac{1}{s+1}\right).
 \label{universal-system}
\end{equation}
The matrix entries of $\rho_N(a_N(s))$ belong to $\kf(s)$.  Indeed, the
exponential in \eqref{xN} becomes a finite sum in a unipotent
representation.

We equip $\kf(s)\otimes\mathcal O(\GG_N)$ with the difference structure
which shifts scalar coefficients by $s\mapsto s+1$ and acts on a
coordinate function $f\in\mathcal O(\GG_N)$ by
\[
 (\sigma f)(g)=f(a_N(s)g).
\]

For $M\geq1$, consider the ordered product
\begin{equation}
 \Phi_N(M)=x_N\left(\frac1M\right)
 x_N\left(\frac1{M-1}\right)\cdots x_N(1).
 \label{Phi-product}
\end{equation}
We set $\Phi_N(0)=1$.  The product satisfies
\begin{equation}
 \Phi_N(M+1)=x_N\left(\frac1{M+1}\right)\Phi_N(M).
 \label{Phi-equation}
\end{equation}
Thus, $\rho_N(\Phi_N)$ is a fundamental solution of
\eqref{universal-system} in the ring of matrix-valued sequences.

Write
\[
 \Phi(M)=X(1/M)\cdots X(1),\qquad \Phi(0)=1,
\]
for the untruncated product.  Its image in $\GG_N$ is $\Phi_N(M)$.

\begin{prop}
Before passing to a finite-weight quotient, the value of the coordinate word
$w=y_{m_1}\cdots y_{m_r}$ on the ordered product \eqref{Phi-product} is
\begin{equation}
 w(\Phi(M))=H_w(M).
 \label{coordinate-harmonic}
\end{equation}
Consequently, evaluation on $\Phi_N$ defines a difference ring morphism
\begin{equation}
 \theta_N\colon\kf(s)\otimes\QSh^{\leq N}\longrightarrow\Seq(\kf).
 \label{theta-N}
\end{equation}
\end{prop}

\begin{proof}
In each factor $X(1/n)$ one may select either $1$ or one letter $X_m/n^m$.
To obtain the word $X_{m_1}\cdots X_{m_r}$ from the product with the larger
indices on the left, one must choose
$M\geq n_1>\cdots>n_r\geq1$.  Its coefficient is precisely
\eqref{harmonic-sum}.  Compatibility with the product follows also from
\eqref{harmonic-character}, and compatibility with the shift follows from
\eqref{Phi-equation}.
\end{proof}

\begin{rem}
\label{rem:Joyner}
In the present notation, Joyner's Hurwitz-polyzeta generating series
\cite[Section~1]{Joyner} satisfies
\[
 \mathcal H(z+1)=X(z^{-1})^{-1}\mathcal H(z).
\]
The continuous algebra homomorphism $\jmath$ determined by
\[
 \jmath(X_m)=[t^m]X(t)^{-1}
\]
is a weight-preserving involutive Hopf automorphism.  Indeed, in the free
primitive coordinates it is given by $\jmath(L_m)=-L_m$.
Thus, after setting $z=s+1$ and applying $\jmath$, Joyner's universal
coefficient becomes the coefficient of \eqref{universal-system} in
every finite-weight quotient.  Here we use the finite harmonic-sum
solution \eqref{Phi-product} to obtain a Picard--Vessiot realization
over $\kf(s)$.
Joyner also introduces a group of tensor automorphisms for his category
of unipotent difference connections \cite[Section~2]{Joyner}.
Theorem~\ref{main-theorem} below computes the Picard--Vessiot groups
of the finite-weight systems over the specified rational base field.
\end{rem}

\section{The Picard--Vessiot ring}
\label{sec:PV-ring}

We compute the Galois group by passing to its abelianization.  The
only rational summation obstruction needed is the following lemma.

\begin{lemma}
If $f\in\kf(s)$ and $c_1,\ldots,c_N\in\kf$ satisfy
\[
 f(s+1)-f(s)=\sum_{m=1}^N\frac{c_m}{(s+1)^m},
\]
then $c_1=\cdots=c_N=0$.
\label{rational-pole-obstruction}
\end{lemma}

\begin{proof}
For $\alpha\in\kf$ and $m\geq1$, let $a_{\alpha,m}$ be the coefficient
of $(s-\alpha)^{-m}$ in the principal part of $f$ at $\alpha$, with
$a_{\alpha,m}=0$ when this term is absent.  Only finitely many of these
coefficients are nonzero.  The corresponding coefficient of
$f(s+1)-f(s)$ is $a_{\alpha+1,m}-a_{\alpha,m}$.  Its sum over any
translation orbit $\alpha+\mathbb Z$ is therefore zero.  For the right
hand side of the displayed equation, the sum over $-1+\mathbb Z$ is
$c_m$.  Hence every $c_m$ vanishes.
\end{proof}

\begin{lemma}
Let $G$ be a unipotent algebraic group over $\kf$.  A closed subgroup
$H\subseteq G$ which maps surjectively onto $G^{\mathrm{ab}}$ is equal
to $G$.  Consequently, every proper closed subgroup of $G$ is contained
in the kernel of a nonzero homomorphism $G\to\mathbb G_a$.
\label{unipotent-abelianization}
\end{lemma}

\begin{proof}
Put $\mathfrak g=\operatorname{Lie}(G)$ and
$\mathfrak h=\operatorname{Lie}(H)$, and write
$\gamma_1\mathfrak g=\mathfrak g$,
$\gamma_{r+1}\mathfrak g=[\mathfrak g,\gamma_r\mathfrak g]$.
Surjectivity onto the abelianization gives
$\mathfrak g=\mathfrak h+\gamma_2\mathfrak g$.
If $\mathfrak g=\mathfrak h+\gamma_r\mathfrak g$ for $r\geq2$, then
\[
 \mathfrak g=\mathfrak h+[\mathfrak g,\mathfrak g]
 \subseteq\mathfrak h+
 [\mathfrak h+\gamma_r\mathfrak g,\mathfrak h+\gamma_r\mathfrak g]
 \subseteq\mathfrak h+\gamma_{r+1}\mathfrak g.
\]
Nilpotence implies $\mathfrak h=\mathfrak g$.  Since $G$ is connected
in characteristic zero, this gives $H=G$.
If $H$ is proper, its image in $G^{\mathrm{ab}}$ is thus a proper
linear subspace of an additive vector group.  A nonzero linear
functional vanishing on that image gives the required homomorphism.
\end{proof}

Put
\begin{equation}
 S_N=\kf(s)\otimes\mathcal O(\GG_N),
 \qquad R_N=\theta_N(S_N)\subset\Seq(\kf),
 \label{RN}
\end{equation}
with the difference structure and evaluation morphism of
Section~\ref{sec:universal-system}.

\begin{theorem}
The morphism $\theta_N$ is injective, and $R_N$ is a Picard--Vessiot
ring for the universal system \eqref{universal-system}.  Its difference
Galois group is
\begin{equation}
 \Gal(R_N/\kf(s))=\rho_N(\GG_N).
 \label{universal-Galois-theorem}
\end{equation}
\label{main-theorem}
\end{theorem}

\begin{proof}
Write $K=\kf(s)$ and $G=\GG_N$, identified with its image under the
faithful representation $\rho_N$.  Choose a maximal proper
$\sigma$-ideal $J\subset S_N$ and put $R=S_N/J$.  The images of the
matrix coordinate functions form a fundamental solution $Z\in G(R)$,
and they generate $R$ over $K$.  Thus $R$ is a Picard--Vessiot ring by
the usual construction, and its constants are $\kf$.  This construction
can be carried out with the defining equations of $G$ imposed from the
start, because left multiplication by $a_N(s)\in G(K)$ preserves them;
see \cite[Proposition~4.13 and the proof of Corollary~5.8(2)]{SingerNotes}
and \cite[Section~1.1]{vanDerPutSinger}.

Let $H$ be the Picard--Vessiot group of $R/K$.  Every automorphism has
the form $Z\mapsto Zc$ with $c\in\operatorname{GL}(V_N)(\kf)$.
Since both matrices belong to $G(R)$, their quotient belongs to
$G(\kf)$.  Hence $H$ is a closed subgroup of $G$.

Suppose that $H\ne G$.  By Lemma~\ref{unipotent-abelianization}, there
is a nonzero additive character $\chi\colon G\to\mathbb G_a$ vanishing
on $H$.  The element $u=\chi(Z)$ is fixed by every Picard--Vessiot
automorphism, so $u\in K$ by the fixed-ring theorem
\cite[Corollary~5.15]{SingerNotes}.  On the other hand,
\begin{equation}
 \sigma(u)-u=\chi(a_N(s))
 =\sum_{m=1}^N\frac{c_m}{(s+1)^m},
 \qquad c_m=d\chi(e_m).
 \label{abelianized-obstruction}
\end{equation}
Here $d\chi\colon\operatorname{Lie}(G)\to\kf$ is the differential
of $\chi$ at the identity.  We use the exponential description \eqref{xN}; on the additive
group the exponential is the identity.  The classes of
$e_1,\ldots,e_N$ form a basis of the abelianization of
$\Free_{\leq N}$, so the $c_m$ are not all zero.  This contradicts
Lemma~\ref{rational-pole-obstruction}.  We conclude that $H=G$.

The Picard--Vessiot torsor theorem now gives
\[
 \dim R=\dim G=\dim S_N;
\]
see \cite[Theorem~5.5]{SingerNotes}.  By
\eqref{coordinate-GN}, $S_N$ is a polynomial ring over $K$.
A quotient by a nonzero ideal has smaller Krull dimension, hence
$J=0$.  Thus $S_N$ itself is $\sigma$-simple.  The unital difference
morphism $\theta_N$ has a proper $\sigma$-ideal as its kernel, so this
kernel is zero.

Finally, $Y_N=\rho_N(\Phi_N)$ is a fundamental solution in sequences.
Since $\rho_N$ is a closed embedding, its matrix coordinates generate
$\mathcal O(\GG_N)$, and therefore
\begin{equation}
 R_N=\kf(s)[Y_N,\det(Y_N)^{-1}]
 \simeq\kf(s)\otimes\mathcal O(\GG_N).
 \label{PV-coordinate-ring}
\end{equation}
The determinant is $1$, but it is kept in the formula to match the
standard definition.  The isomorphism identifies $R_N$ with the
Picard--Vessiot ring just constructed.  Its Galois action is right
translation:
\begin{equation}
 Y_N\longmapsto Y_N\rho_N(c),\qquad c\in\GG_N(\kf).
 \label{right-action}
\end{equation}
This proves the theorem.
\end{proof}

\begin{cor}
Let $\pi\colon\GG_N\to H$ be a surjective morphism of unipotent algebraic
groups and let $H\subset\operatorname{GL}(V)$ be faithful.  The system
\begin{equation}
 \sigma(Y)=\pi(a_N(s))Y
 \label{quotient-system}
\end{equation}
has difference Galois group $H$.
\label{quotient-corollary}
\end{cor}

\begin{proof}
Use $\pi^*$ to view $\mathcal O(H)$ as a subalgebra of
$\mathcal O(\GG_N)$ and put
\[
 S=\theta_N(\kf(s)\otimes\mathcal O(H))\subset R_N.
\]
The injectivity of $\theta_N$ identifies $S$ with
$\kf(s)\otimes\mathcal O(H)$.  The surjection $\pi$ is faithfully flat,
so $R_N$ is faithfully flat over $S$.  In this unipotent setting this
also follows by choosing a linear section of the induced map of Lie
algebras: exponential and logarithm give a regular section of $\pi$ as
a morphism of varieties, and hence an isomorphism
$\GG_N\simeq\ker(\pi)\times H$ over $H$.

If $I$ is a nonzero $\sigma$-ideal of $S$, then $IR_N$ is a nonzero
$\sigma$-ideal of $R_N$, so $IR_N=R_N$.  Faithful flatness gives
$I=(IR_N)\cap S=S$.  Thus $S$ is $\sigma$-simple, and its constants
are $\kf$ because it is a difference subring of $R_N$ containing $\kf$.
The matrix $\pi(\Phi_N)$ is a fundamental solution of
\eqref{quotient-system}.  Its entries generate $S$, since the faithful
representation of $H$ is a closed embedding.  Hence $S$ is a
Picard--Vessiot ring.  Right translation identifies its Galois group
with $H$, as in the proof of Theorem~\ref{main-theorem}.
\end{proof}

The compatible finite-weight groups form an inverse system.  Thus,
Theorem~\ref{main-theorem} gives the following precise meaning to the
universal group.

\begin{cor}
The pro-Picard--Vessiot group of the inverse system
\eqref{universal-system} is
\begin{equation}
 \varprojlim_N\GG_N
 =\exp\widehat{\Free}(e_1,e_2,\ldots).
 \label{pro-PV-group}
\end{equation}
\end{cor}

The preceding argument also gives a Galois-theoretic proof of the
positive-index case of the independence theorem of Ablinger and
Schneider \cite{AblingerSchneider}.

\begin{cor}
The algebra over $\kf(s)$ generated by the finite harmonic sums is a
polynomial algebra on the sums $H_w$ indexed by Lyndon words.  Evaluation
identifies it with $\kf(s)\otimes\QSh$, and its constants are $\kf$.
\label{AS-form}
\end{cor}

\begin{proof}
Every finite family of elements of $\QSh$ belongs to
$\QSh^{\leq N}$ for some $N$.  Theorem~\ref{main-theorem} therefore
makes evaluation on the full quasi-shuffle algebra injective.
Hoffman's theorem identifies its polynomial generators.  The assertion
about constants follows from its embedding in $\Seq(\kf)$.
\end{proof}

We give the precise comparison with Ablinger and Schneider's summation
convention.  Both the sign of the merged letter and the shift formula
change when weak inequalities are replaced by strict inequalities.

\begin{prop}
For positive integers $m_1,\ldots,m_r$, put
\begin{equation}
 S_{m_1,\ldots,m_r}(M)
 =\sum_{M\geq n_1\geq\cdots\geq n_r\geq1}
 \frac{1}{n_1^{m_1}\cdots n_r^{m_r}},
 \qquad S_{\varnothing}(M)=1.
 \label{weak-harmonic-sum}
\end{equation}
If $I=(i_1,\ldots,i_k)\vDash r$, let $I[\mathbf m]$ be the $k$-tuple obtained
by replacing each consecutive block of $i_j$ entries of
$\mathbf m=(m_1,\ldots,m_r)$ by their sum.  Then
\begin{align}
 S_{\mathbf m}
 &=\sum_{I\vDash r}H_{I[\mathbf m]},
 \label{weak-to-strict}\\
 H_{\mathbf m}
 &=\sum_{I\vDash r}(-1)^{r-\ell(I)}S_{I[\mathbf m]}.
 \label{strict-to-weak}
\end{align}

Let $\star$ denote the weak quasi-shuffle product
\begin{equation}
 (y_au)\star(y_bv)
 =y_a(u\star(y_bv))+y_b((y_au)\star v)-y_{a+b}(u\star v).
 \label{weak-quasi-shuffle}
\end{equation}
The coarsening operator
\begin{equation}
 T(y_{m_1}\cdots y_{m_r})
 =\sum_{I\vDash r}y_{I[\mathbf m]}
 \label{coarsening-operator}
\end{equation}
is a weight-preserving algebra isomorphism from the weak quasi-shuffle
algebra $(\kf\langle Y\rangle,\star)$ to the strict quasi-shuffle algebra
$(\kf\langle Y\rangle,*)$.  After adjoining $\kf(s)$, it is an isomorphism of
difference algebras.
\label{strict-weak-proposition}
\end{prop}

\begin{proof}
The verification may be made one step at a time.

First, every chain $n_1\geq\cdots\geq n_r$ has a unique set of places at
which the inequality is strict.  Collapsing each block of equal indices adds
the corresponding exponents.  Summing over the possible sets of strict places
gives \eqref{weak-to-strict}.

Second, the possible cuts between consecutive indices form a Boolean lattice.
M\"obius inversion on this lattice gives \eqref{strict-to-weak}.  Thus $T$ is
unitriangular with respect to depth and its displayed inverse has the signs
$(-1)^{r-\ell(I)}$.

Third, splitting a product of two weak sums according to the two outer
indices gives \eqref{weak-quasi-shuffle}: the diagonal has been included in
both weakly ordered regions and is therefore subtracted.  For strict sums the
diagonal is a third disjoint region and has the plus sign in
\eqref{quasi-shuffle}.  Substitution of \eqref{coarsening-operator} in the two
recursions gives
\begin{equation}
 T(u\star v)=T(u)*T(v).
 \label{coarsening-product}
\end{equation}

Fourth, the new outer index in a weak sum gives
\begin{align}
 \sigma(S_{m_1,\ldots,m_r})
 &=S_{m_1,\ldots,m_r}
 +\frac{1}{(s+1)^{m_1}}\sigma(S_{m_2,\ldots,m_r})
 \notag\\
 &=S_{m_1,\ldots,m_r}
 +\sum_{j=1}^r\frac{S_{m_{j+1},\ldots,m_r}}
 {(s+1)^{m_1+\cdots+m_j}},
 \label{weak-shift}
\end{align}
where the last factor is $S_{\varnothing}=1$ when $j=r$.  Applying $T$ and using
\eqref{weak-to-strict} reduces \eqref{weak-shift} to the strict shift
\eqref{harmonic-shift}.  Hence $T\sigma=\sigma T$.

At depths two and three, the two transformations read
\begin{align*}
 S_{a,b}&=H_{a,b}+H_{a+b},
 &H_{a,b}&=S_{a,b}-S_{a+b},\\
 S_{a,b,c}&=H_{a,b,c}+H_{a+b,c}+H_{a,b+c}+H_{a+b+c},\\
 H_{a,b,c}&=S_{a,b,c}-S_{a+b,c}-S_{a,b+c}+S_{a+b+c}.
\end{align*}
Also
\begin{equation}
 H_aH_b=H_{a,b}+H_{b,a}+H_{a+b},
 \qquad
 S_aS_b=S_{a,b}+S_{b,a}-S_{a+b}.
 \label{strict-weak-depth-two-products}
\end{equation}
These formulae check explicitly the coefficients and signs used above.
\end{proof}

Ablinger and Schneider work over a field containing $\mathbb Q$.
Restricting their alphabet to the letters $(1,0,m,1)$ with $m\geq1$
gives their nonalternating harmonic alphabet $A_h$, whose nested sum
is exactly \eqref{weak-harmonic-sum}; see
\cite[(1), (5), and Section~2]{AblingerSchneider}.

After quotienting the formal algebra of weak sums by its quasi-shuffle
relations, they realize the reduced difference ring as a tower of
$\Sigma^*$-extensions without new constants and embed it in sequences.
Their basis sums are algebraically independent over the rational
sequences \cite[Theorem~3 and Corollary~1]{AblingerSchneider}.
Proposition~\ref{strict-weak-proposition} transfers this statement to our
strict convention: the map $T$ preserves weight and shift, and its
inverse \eqref{strict-to-weak} ensures that no relation is lost in
either direction.  Hoffman's theorem then identifies the strict
polynomial generators with Lyndon words.  This recovers precisely
Corollary~\ref{AS-form} from their theorem.  The proof above obtains
this positive-index case directly from the Picard--Vessiot group;
their result also treats the larger cyclotomic alphabets.

\section{The Mellin--KZ difference equation}
\label{sec:Mellin-KZ}

We now specialize the universal system of Section~\ref{sec:universal-system}
to the Mellin transform of the KZ equation \eqref{KZ-intro}.
Let $A$ and $B$ be complex $d\times d$ matrices, let $G(z)$ be a solution
of \eqref{KZ-intro} on $(0,1)$, and put
\[
 F(s)=\int_0^1z^{s-1}G(z)\,dz.
\]
On a domain where this integral converges and $z^s(1-z)G(z)$ vanishes
at both endpoints, integration by parts gives
\[
 0=\int_0^1\partial_z\bigl(z^s(1-z)G(z)\bigr)\,dz
   =(s+A)F(s)-(s+1+A-B)F(s+1).
\]
For simultaneously strictly upper triangular $A$ and $B$, the entries
of $G$ have at most polynomial growth in the endpoint logarithms.
Consequently, the integral and the vanishing boundary terms are valid
for $\operatorname{Re}s>0$.
Thus the Mellin transform satisfies
\begin{equation}
 (s+1+A-B)F(s+1)=(s+A)F(s).
 \label{Mellin-KZ}
\end{equation}
We study this rational difference equation over $\kf(s)$ for
$d\times d$ matrices $A$ and $B$ over $\kf$.
The rational gauge transformation
\begin{equation}
 F(s)=(s+A)^{-1}U(s)
 \label{rational-gauge}
\end{equation}
gives
\begin{equation}
 U(s+1)=D(s)U(s),
 \qquad
 D(s)=(s+1+A)(s+1+A-B)^{-1}.
 \label{unipotent-KZ-system}
\end{equation}

Put
\begin{equation}
 C=B-A,
 \qquad t=\frac1{s+1}.
 \label{C-t}
\end{equation}
Then
\begin{equation}
 D(s)=(1+tA)(1-tC)^{-1}
 =1+tB(1-tC)^{-1}
 =1+\sum_{m\geq1}BC^{m-1}t^m.
 \label{KZ-Xm}
\end{equation}
Thus, the universal variables are specialized by
\begin{equation}
 X_m\longmapsto BC^{m-1}.
 \label{KZ-specialization}
\end{equation}

To examine the universal specialization, we first regard $B$ and $C$
as free noncommuting variables and set $A=B-C$.  We return to fixed
matrices in Theorem~\ref{fixed-matrix-theorem}.

\begin{prop}
The algebra morphism
\begin{equation}
 \iota\colon
 \kf\langle X_1,X_2,\ldots\rangle
 \longrightarrow\kf\langle B,C\rangle,
 \qquad X_m\longmapsto BC^{m-1},
 \label{code-map}
\end{equation}
is injective.  It remains injective after completion with respect to the
weight defined by $\wt(X_m)=m$ and $\wt(B)=\wt(C)=1$.
\label{code-map-proposition}
\end{prop}

\begin{proof}
A word $X_{m_1}\cdots X_{m_r}$ is mapped to
\begin{equation}
 BC^{m_1-1}BC^{m_2-1}\cdots BC^{m_r-1}.
 \label{coded-word}
\end{equation}
The positions of the letters $B$ recover the decomposition into blocks and
hence recover $(m_1,\ldots,m_r)$.  Distinct words have distinct images.
The map preserves the weight, which proves the statement for the
completions.
\end{proof}

Let $L_m(A,B)$ be defined by
\begin{equation}
 \log\left(1+\sum_{m\geq1}BC^{m-1}t^m\right)
 =\sum_{m\geq1}L_m(A,B)t^m.
 \label{KZ-logarithm}
\end{equation}
By \eqref{Lm-formula},
\begin{align}
 L_1(A,B)&=B,\notag\\
 L_2(A,B)&=-BA+\frac12B^2,\notag\\
 L_3(A,B)&=BA^2-\frac12BAB-\frac12B^2A+\frac13B^3.
 \label{KZ-small-L}
\end{align}

\begin{cor}
The homomorphism from the free Lie algebra generated by the $L_m$ to the
commutator Lie algebra of $\kf\langle B,C\rangle$ induced by
\eqref{KZ-specialization} is injective.
\label{formal-faithfulness}
\end{cor}

\begin{proof}
The change of associative generators $(X_m)\leftrightarrow(L_m)$ defined by
the logarithm and exponential is triangular and invertible.
Proposition~\ref{code-map-proposition}
therefore implies that the $L_m(A,B)$ freely generate an associative
algebra.  Their Lie subalgebra is the free Lie algebra on these generators.
\end{proof}

The preceding statement concerns the formal variables $A$ and $B$.  The
following theorem gives an ordinary Picard--Vessiot statement for fixed
matrices.

\begin{theorem}
Suppose that $A$ and $B$ are simultaneously strictly upper triangular
$d\times d$ matrices.  Let
\begin{equation}
 \Lie_{A,B}=\operatorname{Lie}\langle
 L_1(A,B),\ldots,L_{d-1}(A,B)\rangle
 \subset\mathfrak{gl}_d(\kf).
 \label{matrix-Lie-algebra}
\end{equation}
Then the difference Galois groups over $\kf(s)$ of
\eqref{Mellin-KZ} and \eqref{unipotent-KZ-system} are both
\begin{equation}
 \exp(\Lie_{A,B})\subset\operatorname{GL}_d(\kf).
 \label{fixed-matrix-group}
\end{equation}
\label{fixed-matrix-theorem}
\end{theorem}

\begin{proof}
For $d=1$, both residues vanish and \eqref{unipotent-KZ-system} is
trivial; the gauge transformation \eqref{rational-gauge} proves the
claim.  Assume $d\geq2$.

Each $L_m(A,B)$ is homogeneous of degree $m$ in $A,B$.
Thus every Lie word of total weight at least $d$ becomes a sum of
products of at least $d$ strictly upper triangular matrices and
vanishes.  In particular, $L_m(A,B)=0$ for $m\geq d$, and the assignment
\begin{equation}
 e_m\longmapsto L_m(A,B)
 \label{Lie-specialization}
\end{equation}
defines a surjective Lie algebra morphism
\[
 \Free_{\leq d-1}\longrightarrow\Lie_{A,B}.
\]
It integrates to a surjective morphism of unipotent groups
\[
 \GG_{d-1}\longrightarrow\exp(\Lie_{A,B}).
\]
By \eqref{KZ-logarithm}, the image of the universal coefficient is exactly
$D(s)$.  Corollary~\ref{quotient-corollary} therefore gives
\eqref{fixed-matrix-group} for \eqref{unipotent-KZ-system}.  Finally, a
rational gauge transformation over the base field does not change the
Picard--Vessiot group.  Equation \eqref{rational-gauge} gives the result for
\eqref{Mellin-KZ}.
\end{proof}

\begin{example}
Let $N$ be the regular $3\times3$ nilpotent Jordan block and set $A=B=N$.
Then
\[
 D(s)=1+\frac{N}{s+1},\qquad
 \log D(s)=\frac{N}{s+1}-\frac{N^2}{2(s+1)^2}.
\]
Thus $L_1=N$, $L_2=-N^2/2$, and
Theorem~\ref{fixed-matrix-theorem} gives
\[
 \Gal=\exp(\kf N\oplus\kf N^2)\simeq\mathbb G_a^2.
\]
The coefficients of $D(s)=1+\sum_{m\geq1}X_m(s+1)^{-m}$ are
$X_1=N$ and $X_m=0$ for $m>1$.  They generate only the one-dimensional
Lie algebra $\kf N$.  This explains why the Galois theorem uses the
coefficients of $\log D(s)$.
\end{example}

The formal KZ specialization is sufficiently large to realize every
finite-weight universal group by matrices.

\begin{prop}
For every $N$, there are simultaneously strictly upper triangular matrices
$A_N$ and $B_N$ for which the Galois group in
Theorem~\ref{fixed-matrix-theorem} is isomorphic to $\GG_N$.
\label{finite-realization}
\end{prop}

\begin{proof}
Let $V_N$ be the quotient of $\kf\langle B,C\rangle$ by the two-sided ideal
spanned by words of length greater than $N$.  Let $B_N$ and $C_N$ act on
$V_N$ by left multiplication and set $A_N=B_N-C_N$.  With respect to the
filtration by word length, these operators are simultaneously strictly
upper triangular.

Every term of weight greater than $N$ acts by zero on $V_N$, so the
Lie algebra map \eqref{Lie-specialization} factors through
$\Free_{\leq N}$.

By Proposition~\ref{code-map-proposition}, the map from the free associative
algebra on the $X_m$ to $\kf\langle B,C\rangle$ is injective in weights at
most $N$.  Left
multiplication on $V_N$ is faithful on this part, since a nonzero element is
detected by its value on the empty word.  The same holds for the free Lie
algebra generated by the $L_m$.  Thus,
\eqref{Lie-specialization} is faithful on $\Free_{\leq N}$, and
Theorem~\ref{fixed-matrix-theorem} gives the result.
\end{proof}

\section{Transport algebra}
\label{sec:transport}

We now express the universal system of Section~\ref{sec:universal-system}
and its Mellin--KZ specialization from Section~\ref{sec:Mellin-KZ}
in terms of the transport algebra of Deligne and Terasoma,
\cite[Proposition 5.1]{DeligneTerasoma},
\cite[Section 6.2]{Terasoma} and \cite[Section 5]{Markarian}.
After extension of scalars to $\kf$, it is the
weight-completed algebra
\begin{equation}
 W=\kf\langle\!\langle w_1,w_2,\ldots\rangle\!\rangle,
 \qquad
 \Delta_*(w_n)=\sum_{i=0}^n w_i\otimes w_{n-i},
 \quad w_0=1,\quad \wt(w_n)=n.
 \label{transport-coproduct}
\end{equation}
This coproduct describes convolution in the de Rham realization
\cite[Proposition 6.2]{DeligneTerasoma}.  Comparison with
\eqref{dual-coproduct} gives an isomorphism of complete Hopf algebras
\begin{equation}
 \eta\colon W\xrightarrow{\sim}\QSh^\vee,
 \qquad w_n\longmapsto X_n.
 \label{transport-identification}
\end{equation}
In particular, $W$ is the noncommutative transport algebra;
its restricted graded dual is the commutative coordinate algebra.

Put
\begin{equation}
 a_W(s)=1+\sum_{n\geq1}\frac{w_n}{(s+1)^n}.
 \label{transport-coefficient}
\end{equation}
The series is interpreted weight by weight.  It is group-like by
\eqref{transport-coproduct}.

\begin{prop}
For a finite-dimensional continuous left $W$-module
$(V,\rho)$, let $\mathcal D(V)$ be the difference module over $\kf(s)$
defined by
\begin{equation}
 U(s+1)=\rho(a_W(s))U(s).
 \label{transport-system}
\end{equation}
This assignment is an exact faithful tensor functor, where tensor products
of $W$-modules are defined using $\Delta_*$.
\end{prop}

\begin{proof}
Continuity means that all terms of sufficiently large weight act by zero.
Thus, $\rho(a_W(s))$ is a finite sum and is invertible, since its constant
term is $1$ and its remaining part belongs to the image of the nilpotent
augmentation ideal in the corresponding weight truncation.
Module morphisms give constant matrices intertwining the systems, and
extension of scalars to $\kf(s)$ is exact and faithful.  Finally,
\begin{equation}
 \rho_{V\otimes V'}(a_W(s))
 =\rho_V(a_W(s))\otimes\rho_{V'}(a_W(s)),
 \label{transport-tensor}
\end{equation}
because $a_W(s)$ is group-like.  The counit gives the trivial system on
the tensor unit.
\end{proof}

For simultaneously strictly upper triangular matrices $A$ and $B$, the
specialization \eqref{KZ-specialization} defines a continuous
$W$-module by
\begin{equation}
 \rho_{A,B}(w_m)=B(B-A)^{m-1},
 \qquad \rho_{A,B}(a_W(s))=D(s).
 \label{transport-KZ-specialization}
\end{equation}
Thus \eqref{transport-system} specializes to the gauged Mellin--KZ
equation \eqref{unipotent-KZ-system}.  This identifies the two
presentations used below.

\begin{cor}
Let $W_q$ be the homogeneous component of weight $q$, so that
$W=\prod_{q\geq0}W_q$.  Let
$\GG^{\mathrm{un}}_\sigma=\varprojlim_N\GG_N$ be the universal
pro-Picard--Vessiot group of \eqref{universal-system}.  Then
\begin{equation}
 \mathcal O(\GG^{\mathrm{un}}_\sigma)
 \simeq W^\vee,
 \qquad
 W^\vee
 =\bigoplus_{q\geq0}W_q^*.
 \label{transport-Galois}
\end{equation}
\end{cor}

\begin{proof}
Under $\eta$, the image of $a_W(s)$ in $\GG_N$ is $a_N(s)$.
Theorem~\ref{main-theorem} identifies its Galois group with $\GG_N$.
Passing to the inverse limit and taking the restricted graded dual of
\eqref{transport-identification} gives \eqref{transport-Galois}.
\end{proof}

The freeness of the group associated with $W$ is already
contained in its harmonic presentation.  The Galois theorem identifies
this group with the universal group realized by the harmonic-sum
solutions.  This is an algebraic comparison.  It does not identify
$\mathcal D$ with a geometric Mellin functor on the perverse-sheaf
quotient, nor compare its solution fiber functor with vanishing cycles.
Such a comparison must account for the extension at $1$ and the
Betti--de Rham comparison; the harmonic generators used here are not
being identified with logarithms of Betti monodromy operators.

\section{Fibers at $1$ and at infinity}
\label{sec:fibers}

In this section, take $\kf=\mathbb C$.  We work with finite-dimensional continuous
$W$-modules and their constant $W$-linear morphisms.  We use the universal
$W$-system \eqref{transport-system}; its matrix specialization
\eqref{transport-KZ-specialization} is the Mellin--KZ equation
\eqref{unipotent-KZ-system} of Section~\ref{sec:Mellin-KZ}.  Put
\[
 x_W(t)=1+\sum_{m\geq1}w_mt^m.
\]
The change of variable $q=s+1$, with $Y(q)=U(q-1)$, gives
\begin{equation}
 Y(q+1)=\rho(x_W(1/q))Y(q).
 \label{fiber-system}
\end{equation}
Thus the finite fiber used below is at $q=1$, or $s=0$ in the original
coordinate.  We use solutions on the positive integers.

Define
\begin{equation}
 P_M=x_W(1/M)\cdots x_W(1),\qquad P_0=1,
 \qquad \widetilde P_M=e^{-H_1(M)w_1}P_M.
 \label{fiber-products}
\end{equation}
Under \eqref{transport-identification}, $P_M$ is the ordered product
\eqref{Phi-product}, before weight truncation.  Its coefficient at
$w_{m_1}\cdots w_{m_r}$ is $H_{m_1,\ldots,m_r}(M)$.

\begin{prop}
The limit
\begin{equation}
 C_* = \lim_{M\longrightarrow\infty}\widetilde P_M
 \label{harmonic-connection}
\end{equation}
exists weight by weight and is group-like.  In every fixed weight
quotient, for some nonnegative integer $d$, one has
\begin{equation}
 P_M=M^{w_1}e^{\gamma w_1}C_*
       +O\bigl(M^{-1}(1+\log M)^d\bigr),
 \label{product-leading-term}
\end{equation}
where $\gamma$ is Euler's constant.
\end{prop}

\begin{proof}
Since $H_1(M)-H_1(M-1)=1/M$, we have
\[
 \widetilde P_M\widetilde P_{M-1}^{-1}
 =\operatorname{Ad}\bigl(e^{-H_1(M-1)w_1}\bigr)
       \bigl(e^{-w_1/M}x_W(1/M)\bigr).
\]
The expression in the last parentheses is $1+O(M^{-2})$ in each fixed
weight quotient.  Conjugation introduces only a polynomial in
$H_1(M-1)$, since $\operatorname{ad}(w_1)$ is nilpotent in that quotient.
Consequently, the increments differ from $1$ by
$O(M^{-2}(1+\log M)^d)$ for some $d$.  This bound is summable, so the
ordered product converges.  Its tail gives
$\widetilde P_M-C_*=O(M^{-1}(1+\log M)^d)$, after increasing $d$ if necessary.
The constant term of $C_*$ is $1$; in particular, it is invertible.
Formula \eqref{product-leading-term} follows from
$H_1(M)=\log M+\gamma+O(M^{-1})$.

The series $x_W(t)$ is group-like, and $w_1$ is primitive.  Hence every
$\widetilde P_M$ is group-like.  Passing to the limit in each total weight gives
\begin{equation}
 \Delta_*(C_*)=C_*\otimes C_*.
 \label{connection-group-like}
\end{equation}
\end{proof}

Let $\mathcal C_W$ be the category of the continuous $W$-modules just
specified.  For $V\in\mathcal C_W$, let
$\operatorname{Sol}_+(V)$ be the vector space of sequences satisfying
\eqref{fiber-system} on the positive integers.  We denote by $V_1$ the
fiber of the prescribed constant bundle at $q=1$, and by $V_\infty$ a
second copy of $V$, interpreted as the space of regularized leading
coefficients.  Define
\begin{equation}
 \begin{split}
 \operatorname{ev}_1(Y)&=Y(1),\\
 \operatorname{lc}_\infty(Y)
   &=\lim_{M\longrightarrow\infty}
       e^{-H_1(M)\rho(w_1)}Y(M+1).
 \end{split}
 \label{fiber-evaluations}
\end{equation}
If $Y(1)=v$, then $Y(M+1)=\rho(P_M)v$, so the second limit is
$\rho(C_*)v$.  Both maps in \eqref{fiber-evaluations} are isomorphisms.

\begin{prop}
The functors
\[
 \omega_1(V)=V_1,\qquad \omega_\infty(V)=V_\infty
\]
are fiber functors on $\mathcal C_W$.  They have two natural tensor
isomorphisms
\begin{equation}
 T_{\mathrm{alg}},T_{\mathrm{hor}}\colon
       \omega_1\longrightarrow\omega_\infty,
 \qquad
 T_{\mathrm{alg},V}=\operatorname{id}_V,
 \qquad
 T_{\mathrm{hor},V}=\rho(C_*).
 \label{two-comparisons}
\end{equation}
Here $T_{\mathrm{alg}}$ is supplied by the common constant trivialization,
and $T_{\mathrm{hor}}=\operatorname{lc}_\infty\circ
\operatorname{ev}_1^{-1}$ is horizontal transport.  Their ratio is the
tensor automorphism represented by $C_*$.
\end{prop}

\begin{proof}
Both functors are copies of the forgetful functor and are therefore
exact and faithful.  The reference comparison is tensor-compatible by
its definition.  The group-likeness of $x_W(t)$ makes pointwise tensor
products of solutions into solutions of the tensor product system.
The primitivity of $w_1$ makes the normalization in
\eqref{fiber-evaluations} compatible with tensor products.  Equivalently,
\eqref{connection-group-like} gives
$\rho_{V\otimes V'}(C_*)=\rho_V(C_*)\otimes\rho_{V'}(C_*)$.
Every constant $W$-linear morphism commutes with these constructions,
which proves naturality.
\end{proof}

We now identify the comparison matrix.  Let $a,b$ be two free
noncommuting variables.  The assignment $w_m\mapsto a^{m-1}b$ identifies
the underlying algebra $W$ with
$\mathbb C\oplus\mathbb C\langle\!\langle a,b\rangle\!\rangle b$.
Let
\[
 \pi\colon\mathbb C\langle\!\langle a,b\rangle\!\rangle
       \longrightarrow W
\]
be the linear projection which kills nonempty words ending in $a$ and
reads each remaining word in these generators.  Consider
\begin{equation}
 G'(z)=\left(\frac a z+\frac b{1-z}\right)G(z),
 \qquad G_0=G_1\Phi_{\mathrm{KZ}}(a,b),
 \label{comparison-KZ}
\end{equation}
where $G_0=K_0(z)z^a$ and $G_1=K_1(z)(1-z)^{-b}$ are the canonical
solutions with $K_0(0)=K_1(1)=1$.  Here the $K_i(z)$ are holomorphic
coefficientwise at the respective endpoints.  Set
$f_Y=\pi(\Phi_{\mathrm{KZ}}(a,b))$ and
\begin{equation}
 \Gamma_{\mathrm{reg}}(u)=e^{\gamma u}\Gamma(1+u)
  =\exp\left(\sum_{m\geq2}
        \frac{(-1)^m\zeta(m)}m u^m\right).
 \label{regularized-Gamma}
\end{equation}
The Gamma function in this formula is expanded at $u=0$.

The next formula is the generating-series regularization comparison
of Costermans and Hoang Ngoc Minh \cite[Theorem~6]{CostermansMinh},
also expressed by the regularization theorem of Ihara, Kaneko and Zagier
\cite[Section~2, Theorem~1]{IKZ}.  We include a proof to identify its
two sides with the comparison maps \eqref{two-comparisons}.

\begin{prop}
With these conventions, the comparison matrix is
\begin{equation}
 C_*=\Gamma_{\mathrm{reg}}(w_1)^{-1}f_Y.
 \label{Gamma-connection}
\end{equation}
Thus $T_{\mathrm{alg}}^{-1}T_{\mathrm{hor}}$ is represented by the
Gamma-corrected projected KZ associator.
\label{Gamma-comparison-proposition}
\end{prop}

\begin{proof}
Put
\[
 L(z)=(1-z)\sum_{M\geq0}P_Mz^M,\qquad 0<z<1.
\]
For every nonempty word, summing first over $M\geq n_1$ gives
\begin{equation}
 [w_{m_1}\cdots w_{m_r}]L(z)
 =\sum_{n_1>\cdots>n_r\geq1}
       \frac{z^{n_1}}{n_1^{m_1}\cdots n_r^{m_r}}.
 \label{Abel-polylogarithms}
\end{equation}
These are the iterated-integral coefficients of $G_0$ at words ending
in $b$.  Thus $L(z)=\pi(G_0(z))$.  The projection commutes with left
multiplication by powers of $b$, read as powers of $w_1$ in $W$.
Consequently, tangential regularization gives
\[
 f_Y=\lim_{z\longrightarrow1^-}(1-z)^{w_1}L(z).
\]
Substitute \eqref{product-leading-term}.  The elementary Abel limit
\[
 \lim_{z\longrightarrow1^-}
 (1-z)^{1+u}\sum_{M\geq1}M^uz^M=\Gamma(1+u)
\]
is locally uniform for $u$ near zero.  It follows by a Riemann-sum
approximation to $\int_0^\infty e^{-t}t^u\,dt$; differentiation at
$u=0$ gives the same assertion for every logarithmic coefficient.
It therefore applies to the nilpotent element $w_1$ in each weight
quotient.  The error in \eqref{product-leading-term} contributes
$O((1-z)(1+|\log(1-z)|)^d)$, for some $d$, after regularization, and
hence tends to zero.  We obtain
\[
 f_Y=\Gamma(1+w_1)e^{\gamma w_1}C_*,
\]
which is \eqref{Gamma-connection}.
\end{proof}

The Gamma factor also records tensor compatibility.  Put
$x=w_1\otimes1$ and $y=1\otimes w_1$.  Equations
\eqref{connection-group-like} and \eqref{Gamma-connection} imply
\begin{equation}
 \frac{\Gamma_{\mathrm{reg}}(x)\Gamma_{\mathrm{reg}}(y)}
      {\Gamma_{\mathrm{reg}}(x+y)}\,
 \Delta_*(f_Y)=f_Y\otimes f_Y.
 \label{Gamma-tensor-defect}
\end{equation}
This is the beta factor in the formulation of the regularized double
shuffle relations used in \cite{Racinet, Terasoma, Markarian}.
Thus multiplication by $\Gamma_{\mathrm{reg}}(w_1)^{-1}$ removes the
tensor defect of the uncorrected comparison.  The Gamma factor itself
is not a tensor automorphism for the harmonic coproduct.

\begin{rem}
The normalization at infinity is part of the construction.  Replacing
$H_1(M)$ by $\log M$ gives
\[
 C_{\log}=\lim_{M\longrightarrow\infty}M^{-w_1}P_M
          =e^{\gamma w_1}C_*.
\]
More generally, multiplication of $\Gamma_{\mathrm{reg}}(u)$ by
$e^{cu}$ leaves the beta factor in \eqref{Gamma-tensor-defect} unchanged.
The condition that its linear coefficient vanish fixes this ambiguity.
\end{rem}

\begin{rem}
The word convention in \eqref{comparison-KZ} must be distinguished from
the matrix realization \eqref{KZ-specialization}.  Put $C=B-A$ and substitute
$a=C$, $b=B$.  The algebra isomorphism
\[
 \mathbb C\oplus\mathbb C\langle\!\langle C,B\rangle\!\rangle B
 \longrightarrow
 \mathbb C\oplus B\mathbb C\langle\!\langle C,B\rangle\!\rangle,
 \qquad c+hB\longmapsto c+Bh
\]
sends $C^{m-1}B$ to $BC^{m-1}$ and preserves the order of the blocks.
It is the left/right transport-algebra identification of
\cite[Section~5]{DeligneTerasoma}.  The KZ residues $(C,B)$ are obtained
from $(A,B)$ by the change of variable $z=1/t$, which exchanges $0$ and
$\infty$.  Accordingly, identifying this comparison with a connection
matrix for the original KZ equation also requires the corresponding
paths and tangential normalizations.
\end{rem}

The two comparisons above are constructed on $\mathcal C_W$ with its
prescribed constant systems.  In particular, $T_{\mathrm{alg}}$ uses
their common trivialization; it is not defined here for arbitrary
rational presentations of difference modules.  A geometric
interpretation requires identifying both fibers and both comparisons
on the perverse-sheaf quotient of \cite{Markarian}.  The Gamma-object
twist in \cite[Definition~7.14 and Proposition~7.15]{DeligneTerasoma},
which converts convolution into a tensor product, provides a related
construction.  We do not establish that geometric identification here.

\bibliographystyle{alpha}
\bibliography{universal_galois.bib}

\end{document}